\documentclass{amsart}
\usepackage{amsmath,amsthm}
\usepackage{amsfonts,amssymb}
\usepackage{accents}
\usepackage{enumerate}
\usepackage{accents,color}
\usepackage{graphicx}
\usepackage{wrapfig}
\newcommand{\lvt}{\left|\kern-1.35pt\left|\kern-1.3pt\left|}
\newcommand{\rvt}{\right|\kern-1.3pt\right|\kern-1.35pt\right|}

\newtheorem{thm}{Theorem}[section]
\newtheorem{cor}[thm]{Corollary}
\newtheorem{lem}[thm]{Lemma}
\newtheorem{prop}[thm]{Proposition}

\newtheorem{defn}[thm]{Definition}

\theoremstyle{remark}

 \def\la{{\langle}}
 \def\ra{{\rangle}}

 \def\sph{{\mathbb{S}^{d-1}}}

 \def\d{{\mathrm{d}}}
 \def\E{{\mathrm{e}}}
  
 \def\a{{\alpha}}
 \def\b{{\beta}}
 \def\g{{\gamma}}
 \def\k{{\kappa}}
 \def\t{{\theta}}
 \def\l{{\lambda}}
 \def\o{{\omega}}
 \def\s{\sigma}
 \def\la{{\langle}}
 \def\ra{{\rangle}}

 \def\kb{{\mathbf k}}

 \def\xb{{\mathbf x}}

 \def\CH{{\mathcal H}}

 \def\CP{{\mathcal P}}

 \def\CV{{\mathcal V}}

 \def\BB{{\mathbb B}}

 \def\NN{{\mathbb N}}

 \def\RR{{\mathbb R}}
 \def\SS{{\mathbb S}}
 \def\TT{{\mathbb T}}
 \def\VV{{\mathbb V}}

\def\lla{\langle{\kern-2.5pt}\langle}      
\def\rra{\rangle{\kern-2.5pt}\rangle}

\def\f{\frac}

\def\vc#1{\mbox{\boldmath$#1$\unboldmath}}

\graphicspath{{./}}
\begin{document}

\title{Orthogonal Polynomials in and on a Quadratic Surface of Revolution} 

\author{Sheehan Olver}
\address{Department of Mathematics\\
Imperial College\\
 London \\
 United Kingdom  }\email{s.olver@imperial.ac.uk}

\author{Yuan Xu}
\address{Department of Mathematics\\ University of Oregon\\
    Eugene, Oregon 97403-1222.}\email{yuan@uoregon.edu}


\date{\today}

\begin{abstract} 
 We present explicit constructions of orthogonal polynomials inside quadratic bodies of revolution, including cones, hyperboloids, and paraboloids. We also construct orthogonal polynomials on the surface of quadratic surfaces of revolution, generalizing spherical harmonics to the surface of a cone, hyperboloid, and paraboloid. We use this construction to develop cubature and fast approximation methods. 
\end{abstract}

\maketitle

\section{Introduction}
\setcounter{equation}{0}

We consider orthogonal polynomials on the boundary or inside the domain bounded by quadratic surfaces
of revolution in $\RR^{d+1}$ for $d \ge 2$. For $\RR^3$, quadratic surfaces of revolution arise from rotating 
curves on the plane around an axis. The most well--known example is the unit sphere. Other examples 
include the cone with its apex at the origin,
$$
     \Omega = \{(x,y,t): x^2+y^2 = a t^2\}, \qquad a =\tan \t,
$$ 
where $\t$ is the half angle of the cone, the elliptic paraboloid 
$$
     \Omega = \{(x,y,t): x^2+y^2 = a t\}, \qquad 0\le t \le b, \quad a >0,
$$ 
and the circular hyperboloid,
$$
     \Omega  = \{(x,y,t): x^2+ y^2 = a (t^2 + c^2)\}, \qquad a >0,
$$ 
where $t$ is in an interval, either finite or infinite, and $t$ can also be an union of two intervals. We shall 
consider orthogonal polynomials on each quadratic surface and in the domain bounded by such a surface.

The work is a continuation of \cite{OX1, OX2}, in which orthogonal polynomials on a generic quadratic 
curve in $\RR^2$ are studied. When the curve is the unit circle, orthogonal polynomials are trigonometric 
polynomials in angular variable. For other quadratic curves equipped with a fairly general weight function, 
the space of orthogonal polynomials of degree $n$ with respect to the weight function has dimension 2, 
and a basis is constructed explicitly in \cite{OX2}. In the current paper, we are interested in the case of 
quadratic surfaces in $\RR^{d+1}$ for $d \ge 2$. The orthogonal structure in and on these domains is 
far more complicated. Our study is modelled after the structure on the unit sphere and on the unit ball. 

In the case of the unit sphere $\sph = \{\vc x \in \RR^d: \|\vc x\| =1\}$, spherical harmonics are orthogonal with 
respect to the inner product 
$$
  \la f, g\ra_\sph = \frac{1}{\o_d}\int_\sph f(\vc x) g(\vc x) \d \s(\vc x),
$$
where $\d\s$ is the surface measure and $\o_d$ denotes the surface area of $\sph$. On the unit ball
$\BB^d = \{\vc x \in \RR^d: \|\vc x\| \le 1\}$, classical orthogonal polynomials are orthogonal with respect to the 
inner product 
$$
  \la f, g\ra_\mu = b_\mu \int_{\BB^d} f(\vc x) g(\vc x) (1-\|\vc x\|^2)^{\mu - \f12} \d \vc x,
$$
where $\mu > -\f12$ and $b_\mu$ is a normalization constant so that $\la 1, 1 \ra_\mu=1$. 
For each quadratic surface of revolution in $\RR^{d+1}$, we shall define a family of weight functions and
construct an explicit basis of orthogonal basis of polynomials with respect to the inner product defined via
a weight function in the family. Furthermore, we shall do the same for a family of weight functions defined
on the domain bounded by each quadratic surface of revolution. 

As an application, we shall consider approximation of functions by expansion in the orthogonal polynomial basis. This leverages recent progress by Slevinsky on fast spherical harmonics transforms \cite{SlSH}, and its adaption to expansion in orthogonal polynomials on disks and triangles \cite{SlTr}. This is the first step in a wider program of developing spectral element methods on such exotic domains. 

The paper is organized as follows. We recall basic results on orthogonal polynomials in the next section. 
Quadratic surface of revolution are defined in the third section. Orthogonal polynomials on a quadratic 
surface and inside the domain bounded by the surface are discussed in Sections 4 and 5, respectively. 
Product type cubature formulas on these domain are stated in Section 6. Finally, fast algorithm for 
evaluating orthogonal expansions of functions on the cone is provided in Section 7.

\medskip
\noindent
{\bf Acknowledgment}. The second author would like to thank the Isaac Newton Institute for Mathematical 
Sciences, Cambridge, for support and hospitality during the programme ``Approximation, sampling and 
compression in data science" where part of the work on this paper was undertaken. This work is supported
by EPSRC grant no EP/K032208/1. 

\section{Preliminary}
\setcounter{equation}{0}

We recall basics on orthogonal polynomials of several variables. Let $\Omega$ be a domain in $\RR^d$
with positive measure. Let $W$ be a non--degenerate weight function on $\Omega$, that is, $W$ is nonnegative,
has finite moments $\int_\Omega \vc x^\a W(\vc x) \d \vc x$ for all $\a \in \NN_0^d$ and $\int_\Omega W(\vc x) \d \vc x > 0$. 
Define the inner product 
$$
  \la f, g\ra = \int_{\Omega} f(\vc x) g(\vc x) W(\vc x) \d \vc x. 
$$
Then orthogonal polynomials with respect to this inner product are well-defined. Let $\Pi_n^d$ be the space of 
orthogonal polynomials of degree $n$ in $d$-variables. We call a polynomial of degree $P \in \Pi_n^d$ orthogonal 
with respect to the inner product $\la \cdot,\cdot\ra$ if $\la P, Q\ra  = 0$ for all $Q \in \Pi_{n-1}^d$. Let $\CV_n^d$ 
be the space of orthogonal polynomials of total degree at most $n$. Then 
\begin{equation}\label{eq:dimVn}
   \dim \CV_n^d = \dim \CP_n^d = \binom{n+d -1}{d} \quad \hbox{and} \quad \dim \Pi_n^d =  \binom{n+d}{d}.
\end{equation}
where $\CP_n^d$ denote the space of homogeneous polynomials of degree $n$ in $d$-variables. In contrast 
to the case of one-variable, the space $\CV_n^d$ can have many different bases for $d \ge 2$. In particular, 
the elements of the basis of $\CV_n^d$ may not be orthogonal to each other. A basis is called orthogonal if its 
elements are mutually orthogonal and it is called orthonormal if, in additional, $\la P,P\ra =1$ for each element 
$P$ in the basis. 

The above discussion of orthogonal polynomials excludes the unit sphere $\sph$, since it has measure zero. 
Let $w$ be a nondegenerate nonnegative weight function on the sphere. For an inner product defined on
the unit sphere 
$$
    \la f, g\ra = \int_{\sph} f(\vc x) g(\vc x) w(\vc x) \d \vc x,  
$$
orthogonal polynomials are defined in the space $\RR[\vc x] \setminus \la \|\vc x\|^2-1\ra$ or the space of polynomials
restricted on the sphere, since any polynomial that contains a factor $1-\|\vc x\|^2$ will be zero on $\sph$. Let
$\Pi_n(\sph)$ denote the space of polynomials restricted on the unit sphere and let $\CH_n^d(w)$ denote the 
space of orthogonal polynomials of degree $n$ with respect to $\la \cdot,\cdot\ra$. Then \cite[p. 115]{DX}
\begin{equation}\label{eq:dimHn}
\dim \Pi(\sph) = \dim \CP_n^d +\dim \CP_{n-1}^d \quad\hbox{and}\quad
       \dim \CH_n^d(w) = \dim \CP_n^d - \dim \CP_{n-2}^d.  
\end{equation}
Our general discussion on orthogonal polynomials of several variables then applies to this case. In the case 
$w(\vc x) =1$, the orthogonal polynomials are given by spherical harmonics.


A spherical harmonic $Y$ of degree $n$ is an element of $\CP_n^d$ that satisfies $\Delta P = 0$, where $\Delta$
denotes the Laplace operator. Since $Y$ is homogeneous, its value is completely determined by its restriction on 
the unit sphere, since $Y(r \xi ) = r^n Y(\xi)$ if $\vc x = r \xi$, $r > 0$ and $\xi \in \sph$. Explicit orthogonal basis of 
$\CH_n^d$ can be given in terms of the Gegenbauer polynomials in spherical coordinates. In the case $d=2$,
$\dim \CH_n^2 = 2$ for $n \ge 1$, and a basis of $\CH_n^2$ is the classical 
\begin{equation}\label{eq:2dHn}
  Y_1^n (x,y) = r^n \cos n \t \quad \hbox{and} \quad   Y_2^n (x,y) = r^n \sin n \t 
\end{equation}
in the polar coordinates $(x,y) = (r\cos \t,r\sin\t)$. For higher dimensional, a basis of spherical harmonics 
can be given in spherical polar coordinates \cite{DX}.

For orthogonal polynomial on a solid domain, one particular useful example for us is the classical 
orthogonal polynomials on the unit ball $\BB^d$, which are orthogonal with respect to the weight function
\begin{equation}\label{eq:w-ball}
  \varpi_\mu(\vc x)  = (1-\|\vc x\|^2)^{\mu - \f12}, \quad \mu > -\tfrac12.  
\end{equation}
The space $\CV_n^d(\varpi_\mu)$ contains several orthogonal bases that are classical and can be explicitly given. 
In particular, one basis can be given in terms of the Jacobi polynomials and spherical harmonics. For $\a,\b > -1$, 
the Jacobi polynomial $P_n^{(\a,\b)}$ is orthogonal with respect to the weight function
$$
  w_{\a,\b} (t) = (1-t)^\a (1+t)^\b, \qquad -1 < t < 1
$$
and it is uniquely determined by the relation
$$
   c_{\a,\b}\int_{-1}^1 P_n^{(\a,\b)}(t)P_m^{(\a,\b)}(t) w_{\a,\b} (t) dt = h_n^{(\a,\b)} \delta_{n,m},
$$
where $c_{\a,\b}$ is the normalization constant defined by $c_{\a,\b} \int_{-1}^1 w_{\a,\b}(t) dt =1$, and $h_n^{(\a,\b)}$ is given by
$$
       h_n^{(\a,\b)} = \frac{(\a+1)_n (\b+1)_n(\a+\b+n+1)}{n!(\a+\b+2)_n(\a+\b+n+1)}.
$$
We state two orthogonal bases for $\CV_n^d(\varpi_\mu)$. The first basis is given in terms of the Jacobi polynomials and the spherical harmonics
in spherical polar coordinates. For $ 0 \le m \le n/2$, let $\{Y_\ell^{n-2m}: 1 \le \ell \le \dim \CH_{n-2m}^d\}$ be an 
orthonormal basis of $\CH_{n-2m}^d$. Define \cite[(5.2.4)]{DX}

\begin{equation}\label{eq:basisBd}
  P_{\ell, m}^n (\vc x) = P_n^{(\mu-\f12, n-2j+\f{d-2}{2})} \left(2\|\vc x\|^2-1\right)  Y_{\ell}^{n-2m}(\vc x).
\end{equation}
Then $\{P_{\ell,m}^n: 0 \le m \le n/2, 1 \le \ell \le \dim \CH_{n-2m}^d\}$ are an orthogonal basis of 
$\CV_n^d(\varpi_\mu)$. The second basis is given in terms of the Gegenbauer polynomials $C_n^\l(t)$, which
is a constant multiple of the Jacobi polynomial $P_n^{(\l-\f12,\l-\f12)}$ and it is orthogonal with respect to 
$(1-t^2)^{\l-\f12}$. For $\xb \in \RR^d$, let $\xb_j = (x_1,\ldots, x_j)$ for $1 \le j \le d$. Define
$$
  P_{\kb}^n(\vc x) = \prod_{j=1}^d (1-\|\xb_j\|^2)^{\frac{k_j}{2}} C_{k_j}^{\mu+\k_{j+1}+\cdots + \k_d+\frac{d-j}2} \Big(\frac{x_j}{\sqrt{1-\|\xb_j\|^2}}\Big).
$$
Then the polynomials $\{P_\kb^n: |\kb| =n\}$ are an orthogonal basis of $\CV_n^d(\varpi_\mu)$
(\cite[p. 143]{DX}). 

\section{Quadratic surface of revolution}
\setcounter{equation}{0}

For $d \ge 2$, we define quadratic surfaces of revolution in $\RR^{d+1}$ as follows: 

\begin{defn}\label{def:VV}
Let $S$ be an open set of $\RR$ and let $\phi: S \to \RR_+$ be either one of the following: 
\begin{enumerate}[    (i)]
\item a linear polynomial that is nonnegative on $S$;
\item the square root of a nonnegative polynomial on $S$ of degree at most $2$. 
\end{enumerate}
Given such a $\phi$, we define the quadratic domain of revolution $\VV^{d+1}$ in $\RR^{d+1}$ by 
$$
   \VV^{d+1} = \VV_\phi^{d+1} : = \left \{(\vc x,t) \in \RR^{d+1}: \, \|\vc x\| \le |\phi(t)|, \quad  t \in S, \,\, \vc x \in \RR^d \right\} 
$$
and denote its surface, the quadratic surface of revolution, by 
$$
   \partial \VV^{d+1}= \partial \VV_\phi^{d+1} : = \left \{(\vc x,t) \in \RR^{d+1}: \,  \|\vc x\| = |\phi(t)|, \quad  t \in S, 
       \,\, \vc x \in \RR^d \right\}. 
$$
\end{defn} 

In most of our examples, $S$ is an interval $(a,b)$ but it could also be a union of two intervals in some of
our examples. For each given quadratic surface of revolution, we consider orthogonal polynomials for a 
class of weight functions defined either on or inside of the domain. 

\subsection{Cylinder} 
If $\phi$ is a constant, $\phi (t) = r > 0$ on $(a, b)$, then $\VV^{d+1}$ is a {\it cylinder}  
$$
   \VV^{d+1} =  \{(\vc  x,t): \|\vc x\| \le r, \,\, \vc x \in \RR^d, \,\, a \le t \le b\} = \BB^d_r \times [a,b]
$$
in $\RR^{d+1}$. When $d = 2$, this is a cylinder with $t$-axis as its axis.  

\subsection{Cone} 
If $\phi$ is a linear polynomial $\phi (t) = t$ on $(0, b)$, then $\VV^{d+1}$ is a {\it cone}  
$$
   \VV^{d+1} = \{(\vc x,t): \|\vc x\| \le t, \,\, \vc x \in \RR^d, \,\, 0 \le t \le b\}
$$
in $\RR^{d+1}$. When $d = 2$, this is a cone with $t$-axis as its axis and with its apex at the origin in $\RR^3$. 

\begin{figure}[ht]
  \begin{center}
    \includegraphics[width=0.45\textwidth]{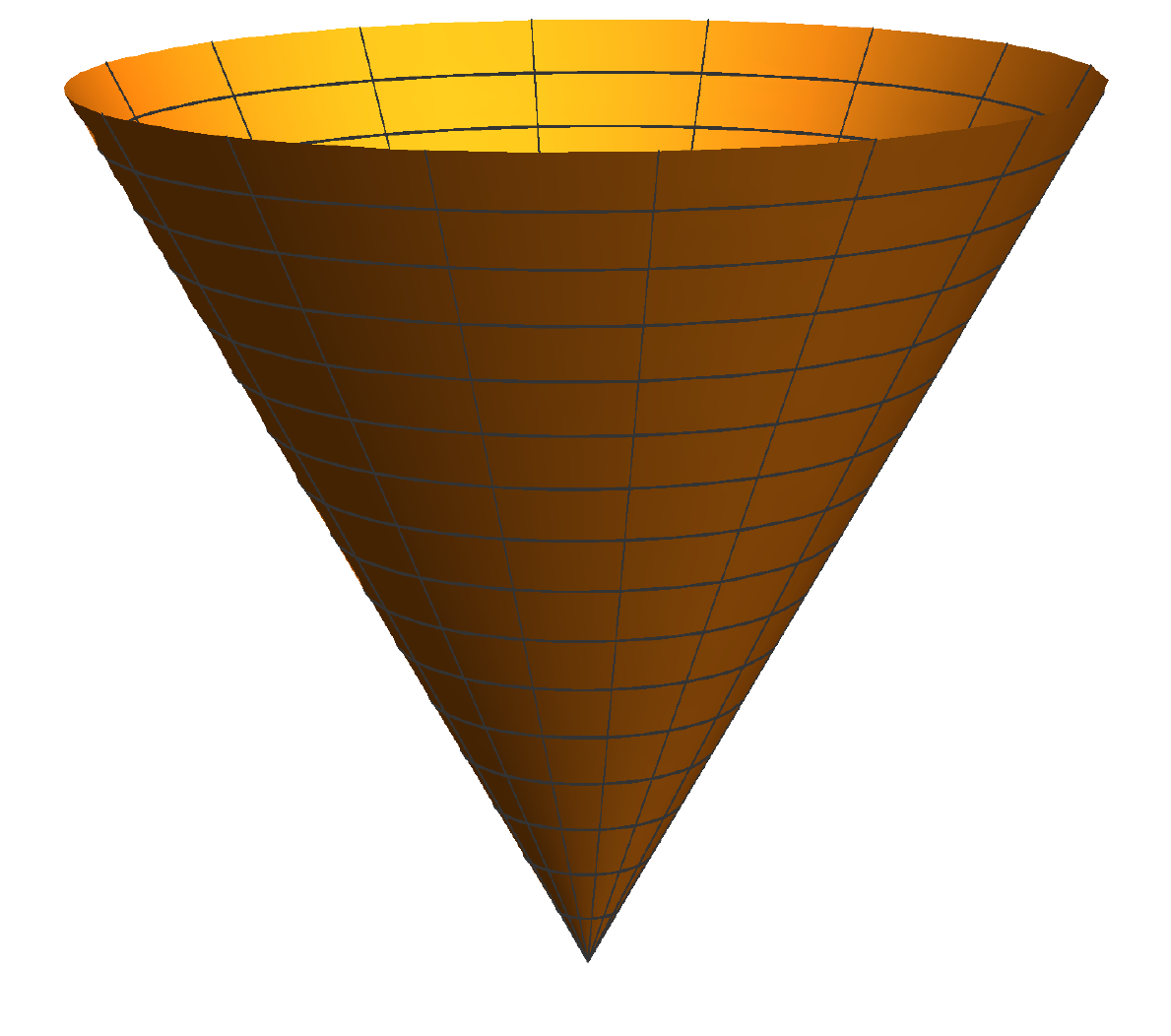} \qquad \quad \includegraphics[width=0.38\textwidth]{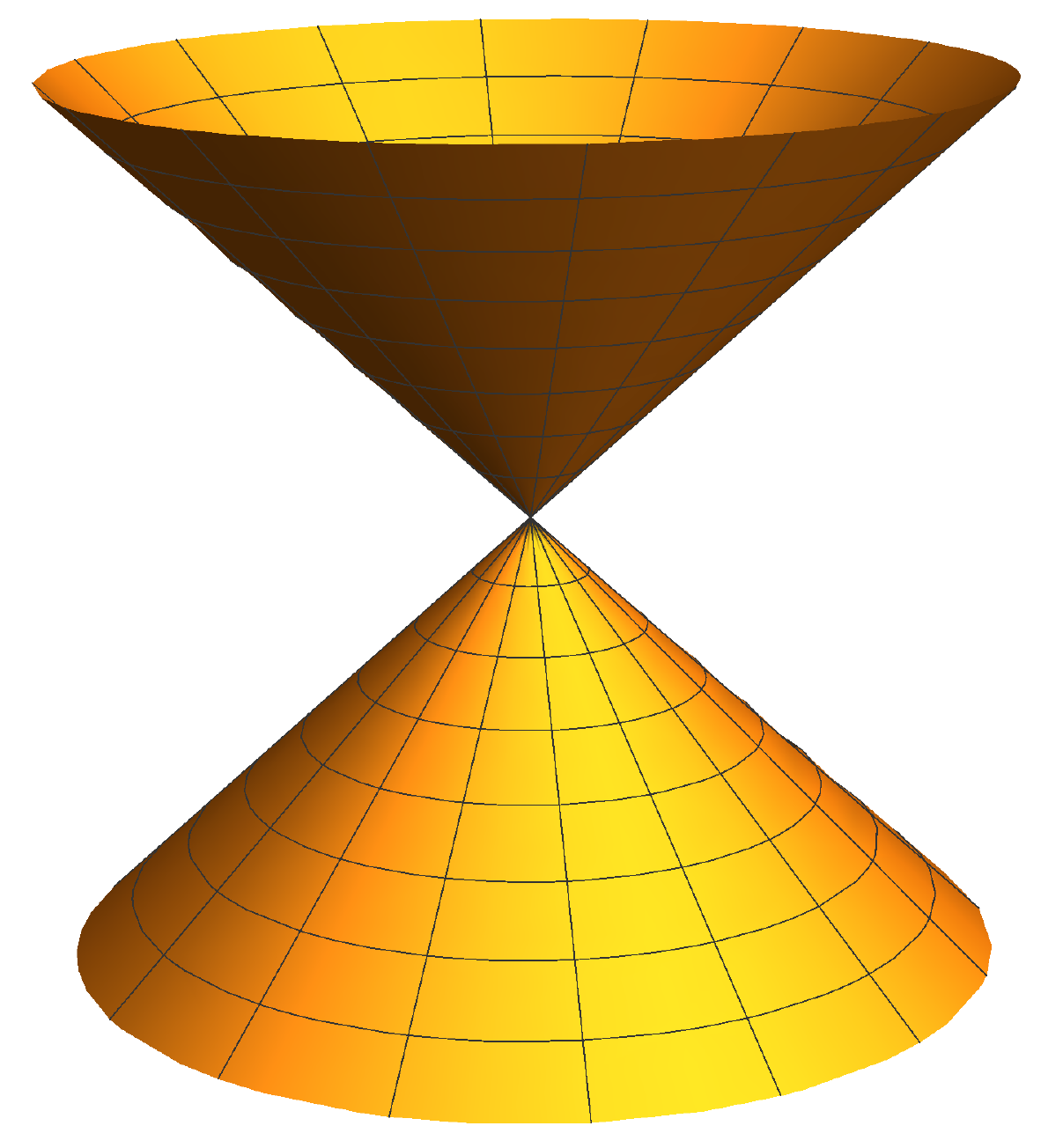} 
  \end{center}
  \caption{Cone and Double Cone}
\end{figure}
 
\subsection{Double cone}
If $\phi (t) = t$ on $(a, b)$ with $a < 0 < b$, then $\VV^{d+1}$ is a {\it double cone}  
$$
   \VV^{d+1} = \{(\vc x,t): \|\vc x\| \le |t|, \,\, \vc x \in \RR^d, \,\, a \le t \le b\}
$$
in $\RR^{d+1}$. When $d = 2$, this is a double cone with $t$-axis as its axis and with its apex at the origin in
$\RR^3$ that extends both above and below $\vc x$-plane.  

\subsection{Ball}
If $\phi$ is given by $\phi(t) = \sqrt{1-t^2}$ for $t \in (0, 1)$, then  $\VV^{d+1}$ is a {\it ball} 
$$
   \VV^{d+1} = \left \{(\vc x,t): \|\vc x\|^2 \le 1-t^2, \,\, 0 \le t \le 1, \,\, \vc x \in \BB^d \right \} = \{(\vc x,t): \|\vc x\|^2+t^2 \le 1\} = \BB^{d+1},
$$
which is the unit ball in $\RR^{d+1}$.  

\subsection{Paraboloid of revolution}
If $\phi$ is given by $\phi(t) = \sqrt{t}$ for $t \in (0, b)$, then  $\VV^{d+1}$ is a {\it paraboloid} 
$$
   \VV^{d+1} = \left \{(\vc x,t): \|\vc x\|^2 \le t, \,\, 0 \le t \le b, \,\, \vc x \in \RR^d \right \}.
$$
When $d = 2$, it is the domain bounded by the rotation of the parabolic curve $y^2 = t$ around $t$-axis 
with $(\vc x,t) = (x,y,t) \in \RR^3$. 
 
\begin{figure}[ht]
  \begin{center}
    \includegraphics[width=0.45\textwidth]{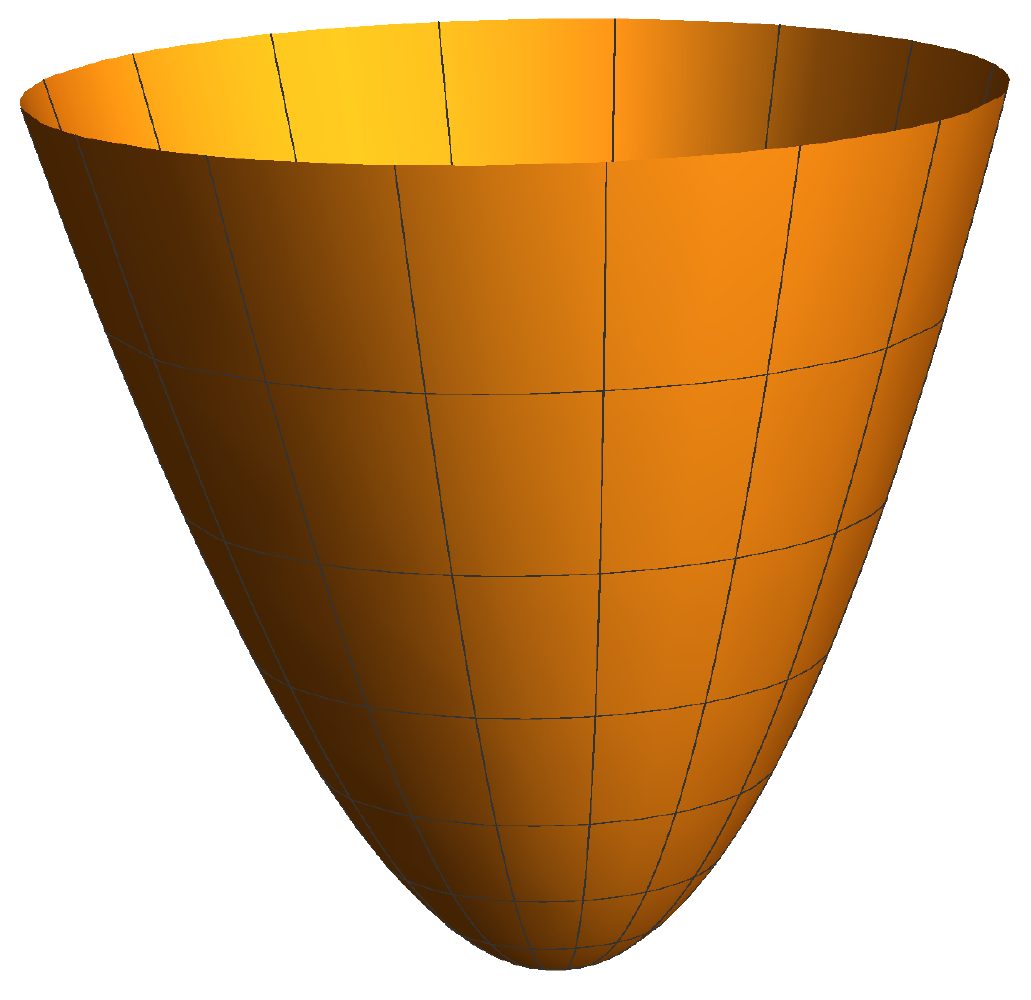} 
  \end{center}
  \caption{Paraboloid and Hyperboloid}
\end{figure}
 
\subsection{Hyperboloid of revolution}
If $\phi$ is given by $\phi(t) = \sqrt{t^2+ \rho^2}$ for $t \in (a, b)$, then  $\VV^{d+1}$ is a {\it hyperboloid} 
$$
   \VV^{d+1} =  \left \{(\vc x,t): \|\vc x\|^2 \le t^2+\rho^2,  \,\, a \le t \le b, \,\, \vc x \in \RR^d\right\}.
$$
When $d = 2$, it is the domain bounded by the rotation of the hyperbolic curve $y^2 - t^2 =\rho^2$ in $\RR^3$
around $t$-axis with $(\vc x,t) = (x,y,t) \in \RR^3$. 
 
\begin{figure}[ht]
  \begin{center}
    \includegraphics[width=0.4\textwidth]{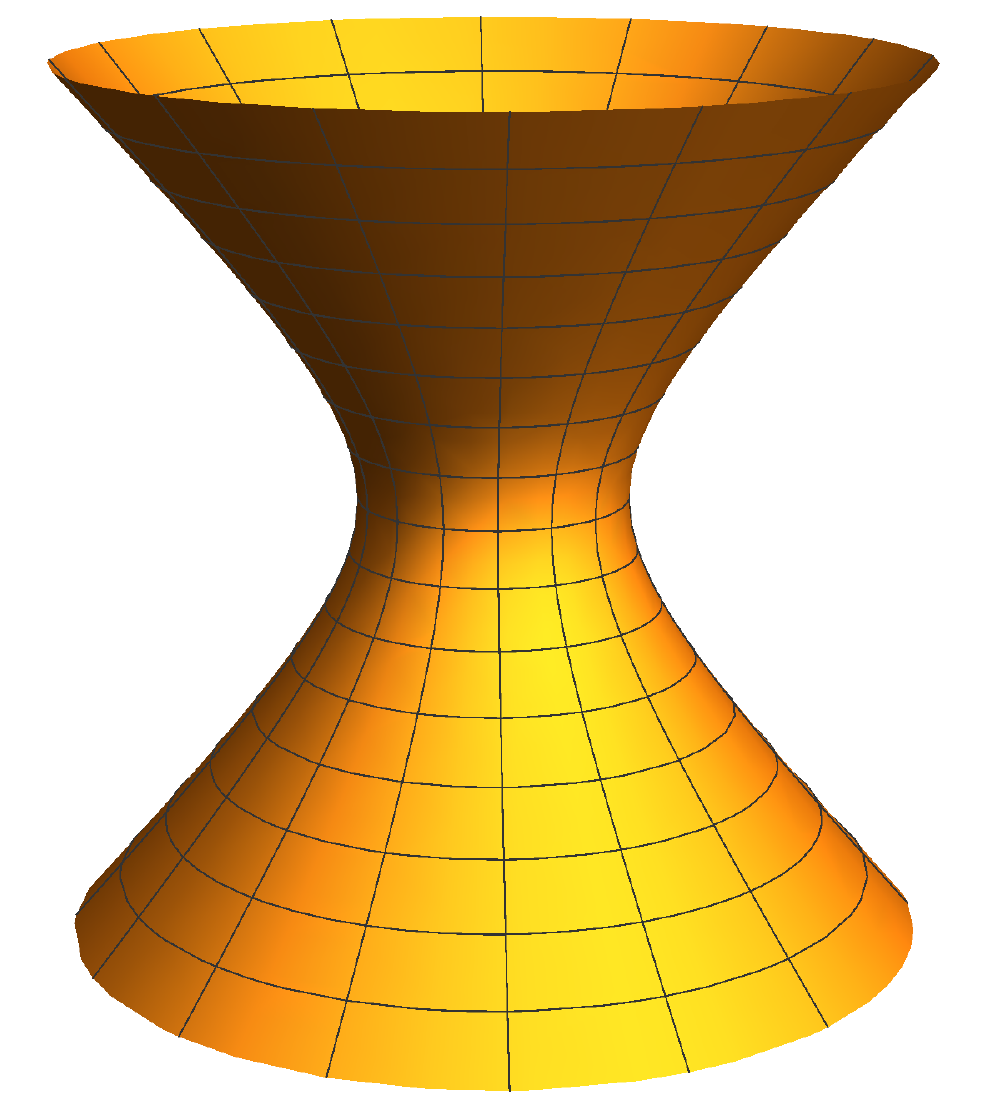} \qquad \quad \includegraphics[width=0.41\textwidth]{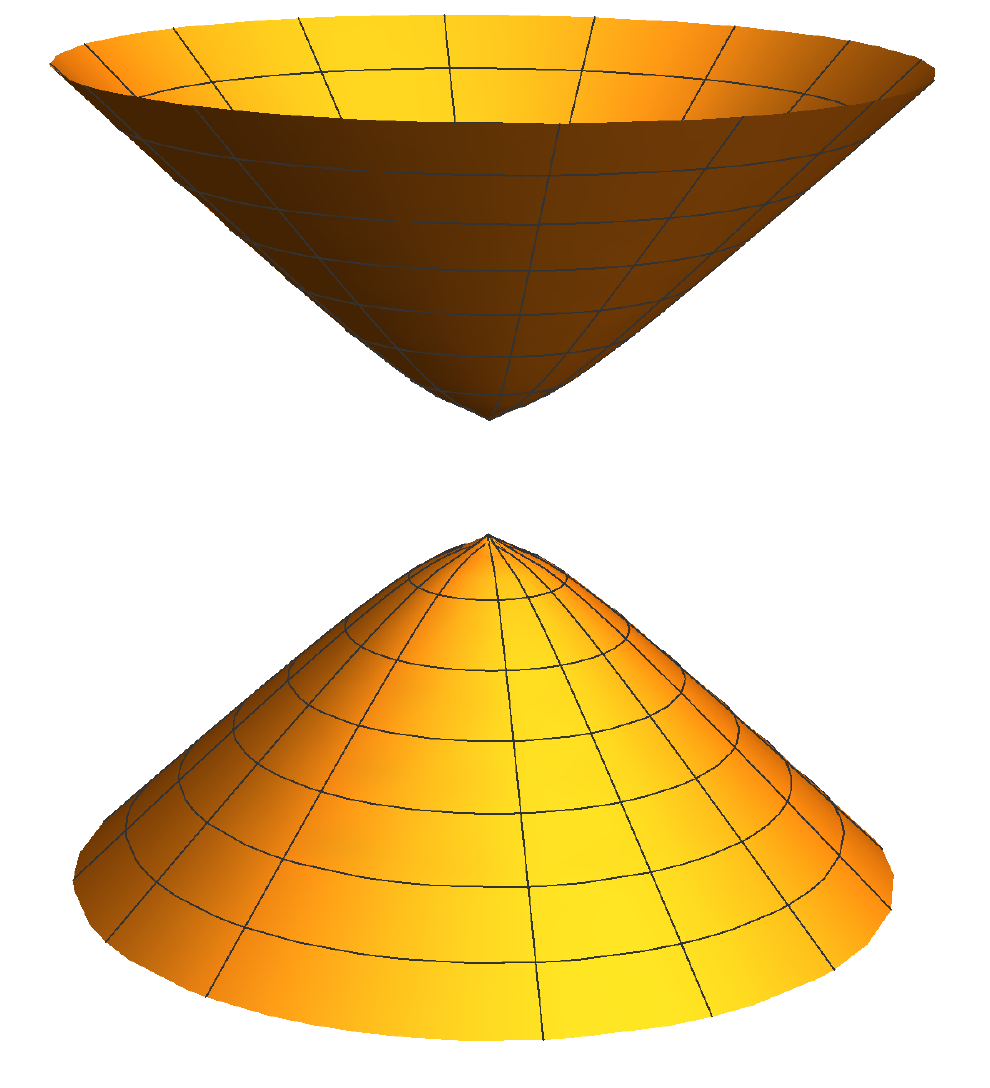} 
  \end{center}
  \caption{Hyperboloid and Hyperboloid of two sheets}
\end{figure}
 
\subsection{Hyperboloid of two sheets}
If $\phi$ is given by $\phi(t) = \sqrt{t^2 - \rho^2}$ for $t \in \{t \in \RR: |t| > \rho\}$, then $\VV^{d+1}$ is a 
{\it hyperboloid of two sheets}, 
$$
   \VV^{d+1} =  \left \{(\vc x,t): \|\vc x\|^2 \le t^2 - \rho^2,  \,\, |t| \ge \rho > 0, \,\, \vc x \in \RR^d\right\}.
$$
When $d = 2$, it is the domain bounded by the rotation of the hyperbolic curve $y^2 - t^2 = - \rho^2$ in $\RR^3$
around $t$-axis with $(\vc x,t) = (x,y,t) \in \RR^3$. 

\section{Orthogonal structure on a quadratic  surface of revolution}
\setcounter{equation}{0}

We consider orthogonal structure with respect to the quadratic surface of revolution. Let $\phi$ 
be a function that satisfies either (1) or (2) of the Definition \ref{def:VV} and let $\VV^{d+1} = \VV^{d+1}_\phi$. 
We define a bilinear form
$$
   \la f,g \ra_w : = c_w \int_{\partial\VV^{d+1}} f(\vc x,t) g(\vc x,t) w(t) \d \sigma(\vc x,t),
$$
where $d\sigma(\vc x,t)$ is the Lebesgue measure on $\partial\VV^{d+1}$ and $w$ is a nonnegative function
defined on the real line such that $\int_{S} |\phi(t)|^{d-1} w(t) \d t < \infty$ and $c_w$ is a normalized constant 
so that $\la 1,1\ra =1$. By choosing the support of $w$, the domain of the integral can be a truncated surface
with $t \in (a,b)$, such as the surface of a finite cone when $\phi(t) = t$ with $a > 0$ or the surface of a 
double cone when $\phi(t) =t$ and $a = -b$. 

Let $\d \s$ denote the surface measure on the unit sphere $\sph$. For $(\vc x,t) \in\partial\VV^{d+1}$, we let 
$\vc x = t\xi$ with $\xi \in \sph$. Then
\begin{align} \label{eq:int-surface}
  \int_{\partial\VV^{d+1}} f(\vc x,t) w(t) d\sigma(\vc x,t) \, & = \int_{\RR} \int_{\|\vc x\| =|\phi(t)|} f(\vc x,t) \d \sigma(\vc x) w(t) \d t \\
      & = \int_{S} |\phi(t)|^{d-1} \int_{\sph} f(\phi(t) \xi, t) d\sigma(\xi) w(t) dt. \notag
\end{align}
This identity implies immediately that the bilinear form $\la \cdot,\cdot\ra$ is an inner product on the polynomial 
space $\RR[\vc x,t]  / \la \|\vc x\|^2 - \phi(t)^2\ra$. 

For $n =0,1,\ldots$, let $\Pi_n(\partial \VV^{d+1})= \Pi_n^{d+1} /  \la \|\vc x\|^2 - \phi(t)^2\ra$ denote the restriction
of polynomials of degree at most $n$ in $d+1$ variables on the boundary of $\VV^{d+1}$. 

\begin{prop} 
For $n =0,1,2,\ldots$, 
\begin{equation} \label{eq:dimSurface}
   \dim \Pi_n(\partial\VV^{d+1}) =  \dim \Pi_n^d+ \dim \Pi_{n-1}^d = \binom{n+d}{d}+\binom{n+d-1}{d}.
\end{equation}
\end{prop}

\begin{proof}
Let $P$ be a polynomial of degree $n$ in $(\vc x,t)$ variable. We can write 
$$
  P(\vc x,t) = \sum_{j=0}^n P_{n-j}(\vc x) t^j, \qquad P_{n-j} \in \Pi_{n-j}^d.
$$
Since $\phi(t)^2$ is a polynomial of degree $2$ in $t$, say $a \phi(t)^2 = t^2 + b t + c$ with $a \ne 0$,
we can write  
$$
t^2 = a(\phi(t)^2 - \|\vc x\|^2) + a \|\vc x\|^2 -  b t -c,
$$ 
which implies, by induction in $j$, that for $j \ge 2$,
$$
     t^j = a^j (\phi(t)^2 - \|\vc x\|^2) R_j(\vc x,t) + S_j(\vc x) + t T_j(\vc x), 
$$
where $R_j$ is a polynomial of degree $j-2$, $S_j$ and $T_j$ are polynomials of degree $j$ and $j-1$, 
respectively. Consequently, we conclude that 
$$
   P(\vc x,t) = S(\vc x) + t T(\vc x) + (\phi(t)^2 - \|\vc x\|^2) R(\vc x,t),
$$
where $S \in \Pi_n^d$ and $T \in \Pi_{n-1}^d$. Since any $S(\vc x)$ in $\Pi_n^d$ and $t T(\vc x)$ with $T \in \Pi_{n-1}^d$
are elements of $\Pi_n(\VV^{d+1})$, the above displayed formula proves \eqref{eq:dimSurface}.
\end{proof}

\begin{cor} 
Let $\CV_n(\partial\VV^{d+1})$ be the space of orthogonal polynomials of degree $n$ with respect to 
the inner product $\la \cdot,\cdot \ra_w$. Then
$$
   \dim \CV_n(\partial\VV^{d+1}) =  \binom{n+d}{n}+\binom{n+d-1}{n-1} = \dim \Pi_n(\SS^d).
$$
\end{cor}

\begin{proof}
Using the Gram--Schmidt process, the dimension of $\CV_n(\partial\VV^{d+1})$ is equal to the dimension of
$\Pi_n(\partial \VV^{d+1})\setminus \Pi_{n-1}(\partial \VV^{d+1})$, which is equal to $\dim \CP_n^d + 
\dim\CP_{n-1}^d$ by the proof of the proposition. 
\end{proof}

We give an orthogonal basis of $\CV_n(\partial\VV^{d+1})$. For a fixed $m \in \NN_0$, let $q_n^{(m)}$ 
be the orthogonal polynomial of degree $n$ with respect to the weight function $\phi(t)^{2 m + d-1} w(t)$ on 
$S$. Let $\{Y_\ell^m: 1 \le \ell \le \dim \CH_n^d\}$ denote an orthonormal basis of $\CH_m^d$. We define
\begin{equation} \label{eq:sfOPbasis}
  S_{m, \ell}^n (x,y) = q_{n-m}^{(m)} (t) \phi(t)^m Y_\ell^m \left(\frac{x}{\phi(t)}\right), 
      \quad 0 \le m \le n, \,\, 1 \le \ell \le \dim \CH_m^d.
\end{equation}

\begin{thm} 
The polynomial $S_{m,\ell}^n$ is an element of $\Pi_n(\partial \VV^{d+1})$. Furthermore, 
the set $\{S_{m,\ell}^n: 0 \le m \le n, 1 \le \ell \le \dim \CH_n^{d}\}$ is an orthogonal basis of 
$\CV_n(\partial\VV^{d+1})$ and
$$
\la S_{m, \ell}^n, S_{m', \ell'}^{n'} \ra_w = h_m^n  \delta_{n,n'} \delta_{m,m'} \delta_{\ell,\ell'},
$$
where 
$$
   h_m^n = c_w \omega_{d} \int_S [q_{n-m}^{(m)}]^2 |\phi(t)|^{2m + d-1} w(t) dt > 0.
$$
\end{thm}

\begin{proof}
Since $Y_\ell^m$ are homogeneous polynomials, it is easy to see that $S_{m,\ell}^n$ is a polynomial of
degree $n$ in $(\vc x,t)$. Using \eqref{eq:int-surface}, we obtain
\begin{align*}
 \la S_{m, \ell}^n, S_{m', \ell'}^{n'} \ra_w &\, = c_w \int_{\partial\VV^{d+1}}
      S_{m, \ell}^n (\vc x,t)S_{m', \ell'}^{n'} (x,y) w(t) \d \sigma(\vc x,t) \\
   \, & = c_w   \int_S q_{n-m}^{(m)}(t) q_{n'-m'}^{(m')}(t) |\phi(t)|^{2m +d-1} w(t) \d t
         \int_{\sph} Y_\ell^m(\xi)Y_{\ell'}^{m'}(\xi) \d \s(\xi) \\
   \, &= h_m^n  \delta_{n,n'} \delta_{m,m'} \delta_{\ell,\ell'}.  
\end{align*}
Thus, $\{S_{m,\ell}^n\}$ is an orthogonal set and they form a basis for $\CV_n(\partial\VV^{d+1})$. 
\end{proof}

Below we give several examples for specific conic surface of revolution. We shall always assume that
$\{Y_\ell^m: 1 \le \ell \le \dim \CH_n^d\}$ denotes an orthonormal basis of $\CH_m^d$. 

\subsection{Surface of a cone} We choose $\phi(t) = t$. When $w = u_{\a,\b}$ is the Jacobi weight function 
supported on $(0,1)$,  
\begin{equation} \label{eq:Jacobi01}
  u_{\a,\b}(t) = t^\a (1-t)^\b \chi_{(0,1)}(t) \quad \a > -d, \quad \b > -1,
\end{equation}
the orthogonal polynomials are given in terms of the Jacobi polynomials by 
$$
   S_{m,\ell}^n(\vc x,t) = t^m P_{n-m}^{(\a+2m +d-1,\b)}(1-2t) Y_\ell^m(\vc x), \qquad 0\le m \le n, \quad 1 \le \ell \le \dim \CH_n^d,  
$$
which are orthogonal on the surface of the finite cone 
$$
\partial\VV^{d+1}: =\{(\vc x,t): \vc x \in \RR^d, 0 \le t \le 1, \, \|\vc x\| =t\}.
$$ 
In particular, for $d =2$, the domain is the surface in $\RR^3$ and the orthogonal basis is given 
in terms of the coordinates $(x,y,t) = (t\cos \t, t \sin \t, t)$ by 
\begin{align}  \label{eq:SurfaceConeDef}
\begin{split}
    S_{m,1}^n(x,y,t) &\, = P_{n-m}^{(\a+2m +1,\b)}(1-2t) t^m \cos m \t, \\
    S_{m,2}^n(x,y,t)&\, = P_{n-m}^{(\a+2m +1,\b)}(1-2t) t^m \sin m \t.  
\end{split}
 \end{align}
When $\a = -1$, these polynomials are also eigenfunctions of a second order differential operators \cite{X}. 

Another example is when $w(t) = (1-t^2)^{\mu-\f12}$ on the interval $(-1,1)$, then the orthogonal polynomials
are given by 
$$
  S_{m,\ell}^n(\vc x,t) = C_{n-m}^{(\mu, m + \f{d-1}{2})}(t) Y_\ell^m(\vc x), \qquad 0\le m \le n, \quad 1 \le \ell \le \dim \CH_n^d,
$$
where $C_n^{(\mu,\nu)}(t)$ is the generalized Gegenbauer polynomial of degree $n$, which is orthogonal with
respect to the weight function
\begin{equation}\label{eq:GGweight}
  w_{\mu,\nu}(t) = |t|^{2\nu} (1-t^2)^{\mu-\f12}, \qquad t \in (-1,1),
\end{equation}
and these polynomials are orthogonal on the surface of the double cone 
$$
\partial\VV^{d+1}: =\{(\vc x,t): \vc x \in \RR^d,  -1 \le t \le 1, \, \|\vc x\| =|t| \}.
$$ 

\subsection{Paraboloid of revolution} 
We choose $\phi(t) = \sqrt{t}$ for $t \in (0, b)$. When $w = u_{\a,\b}$, the
orthogonal polynomials are given in terms of the Jacobi polynomials by 
$$
   S_{m,\ell}^n(\vc x,t) = P_{n-m}^{(\a+m +\frac{d-1}{2},\b)}(1-2t) t^{\f{m}{2}}Y_\ell^m\left(\f{\vc x}{\sqrt{t}}\right), 
         \quad 0\le m \le n, \quad 1 \le \ell \le \dim \CH_n^d,
$$
which are orthogonal on the surface of the paraboloid of revolution
$$
\partial\VV^{d+1}: =\left \{(\vc x,t): \|\vc x\|^2 = t, \,\, 0 \le t \le 1, \,\, \vc x \in \RR^d \right \}.
$$ 
 
\subsection{Hyperboloid of revolution}
We choose $\phi(t) = \sqrt{t^2+ \rho^2}$ for a fixed $\rho \ne 0$.  If we choose $w(t) = u_{\a,\b}$, then 
$q_n^{(m)}$ needs to be orthogonal with respect to the weight function 
$$
    (t^2+ \rho^2)^{m + \frac{d-1}{2}} t^\a (1-t)^\b.
$$
The orthogonal polynomials for this weight function are well-defined but cannot be written explicitly in
terms of classical orthogonal polynomials of one variable.   

\medskip

\section{Orthogonal structure in the domain of revolution}
\setcounter{equation}{0}

We consider orthogonal structure on the quadratic domain bounded by the conic surface of revolution. 
Let $\phi$ be a function that satisfies either (1) or (2) of the Definition \ref{def:VV} and let $\VV^{d+1} 
= \VV^{d+1}_\phi$. 

Let $w_1$ be a weight function defined on $\RR$ and let $w_2$ be a central-symmetric weight function 
defined on the unit ball $\BB^d =\{\vc x: \|\vc x\| \le 1\}$ of $\RR^d$. For each quadratic domain of revolution, 
$\VV^{d+1}_\phi$, we define 
\begin{equation} \label{eq:weightVV}
   W (\vc x,t) : = w_1(t) w_2 \left(\frac{\vc x}{\phi(t)} \right), \qquad  (\vc x,t) \in \VV^{d+1},
\end{equation}
and its associated inner product 
$$
     \la f,g\ra_W: = b_W \int_{\VV^{d+1}} f(\vc x,t) g(\vc x,t)  W(\vc x,t) \d \vc x \d t, 
$$
where $b_W$ is a normalization constant such that $\la 1,1\ra_W = 1$. By choosing the support of $w_1$, 
the domain of the integral can be the the conic solid of revolution with $t \in (a, b)$. For example, it could be
a truncated cone when $\phi(t) = t$ with $0 \le t \le b$.   

When $d =1$, such weight functions and associated orthogonal polynomials of two variables were first 
considered by Koornwinder \cite{K}; see \cite[\S2.6.1]{DX}. It includes the weight function
$(1-x^2-y^2)^{\mu -\f12}$ on the unit disk $\BB^2$ and the weight function
$x^\a y^\b(1-x-y)^\g$ on the triangle $\TT^2 = \{(x,y): x \ge 0, y \ge 0, x+y \le 1\}$. Furthermore,
by extending the domain to 
$$
   \{(x,y): a < x <y, \quad - \phi(x) < y < \phi(x)\}
$$
it also include the weight function $(1-x)^\b (x-y^2)^\a$ on a parabolic biangle 
$R=\{(x,y): y^2<x<1\}$ bounded by a straight line and a parabola. 

For $n \in \NN_0$, let $\CV_n(\VV^{d+1})$ be the space of orthogonal polynomials of degree $n$ 
with respect to this inner product. Then
$$
\dim \CV_n(\VV^{d+1}) = \binom{n+d+1}{n}.
$$ 
We give an orthogonal basis for $\CV_n(\VV^{d+1})$. Let $q_n^{(m)}$ denote an orthogonal polynomial of 
degree $n$ with respect to the weight function $|\phi(t)|^{2m+ d} w_1(t)$ on $S \subset \RR$ and let 
$\{P_\kb^n: |\kb| = n, \kb \in \NN_0^d\}$ denote an orthonormal basis for $\CV_n^d(w_2)$ on $\BB^d$. 
We define 
\begin{equation} \label{eq:OPbasis}
  Q_{m,\kb}^n(\vc x,t) = q_{n-m}^{(2m+d)}(t) \phi(t)^m P_\kb^m\left(\frac{\vc x}{\phi(t)}\right).
\end{equation}

\begin{thm}
The polynomial $Q_{m,\kb}^n$ is an element of $\CV_n(\VV^{d+1})$. Furthermore, the set 
$\{Q_{m,\kb}^n: |\kb| = m \le n, \kb \in \NN_0^d, m =0,1,\ldots, n\}$ is an orthogonal basis 
of $\CV_n(\VV^{d+1})$, 
$$
  \la Q_{m,\kb}^n, Q_{m',\kb'}^{n'} \ra_W  = h_{m}^n \delta_{n,n'} \delta_{m,m'}\delta_{\kb,\kb'}
$$
where, assuming $w_2$ has unit integral on $\BB^d$, 
$$
  h_m^n  = b_W   \int_{S}  [q_{n-m}^{(2m+d)}(t)]^2 [\phi(t)]^{2m +d} w_1(t)\d t.
$$
\end{thm}

\begin{proof}
If $\phi$ is a linear polynomial, it is evident that $\phi(t)^m P_\kb^m\left(\frac{\vc x}{\phi(t)}\right)$ is a polynomial
of degree $m$ in $(\vc x,t)$. Since $w_2$ is centrally symmetric, the polynomial $P_\kb^m(\vc x) = 
\sum_{|\b| \le m} a_{\b,\kb} \vc x^\b$ is a sum of even monomials if $m$ is even and a sum of odd monomials 
if $m$ is odd, so that we can rewrite it as 
$$
\phi(t)^m P_\kb^m\left(\frac{\vc x}{\phi(t)}\right) = \sum_{0 \le k \le m/2} \phi(t)^{2k} \sum_{|\b| = m-2 k} b_{\b,\kb} \vc x^\b. 
$$
If $\phi$ is the square root of a quadratic polynomial, then the above expression shows that
$\phi(t)^m P_\kb^m\left(\frac{\vc x}{\phi(t)}\right)$ is a polynomial of degree $m$. Hence, $Q_{m,\kb}^n$ is a 
polynomial of degree $n$ for each $\a$ and $0 \le m \le n$. For $(\vc x,t) \in \VV^{d+1}$, we make a change of variable 
$\vc x  \to \vc u \in \BB^d$ by $\vc x = \phi(t) \vc u$, so that $\d \vc x = |\phi(t)|^d \d \vc u$ and 
\begin{align}\label{eq:Int-VV}
  \int_{\VV^{d+1}} f(\vc x,t) W(\vc x,t) d\vc x dt \, & = \int_S \int_{\|\vc x\| \le |\phi(t)|} f(\vc x,t) w_1(t) w_2 \left(\frac{\vc x}{\phi(t)}\right)\d \vc x\, \d t \\
      & = \int_S |\phi(t)|^{d} \int_{\BB^d} f(\phi(t) \vc u, t) w_1(t) w_2(\vc u) \d \vc u \, \d t. \notag
\end{align}
The above parametrization gives, since $P_\a^n$ is orthonormal, 
\begin{align*}
 &  \la Q_{m, \kb}^n,  Q_{m', \kb'}^{n'} \ra_W  = c_W \int_{\VV^{d+1}}
      Q_{m, \kb}^n (\vc x,t)Q_{m', \kb'}^{n'} (\vc x,y)  w_1(\phi (t) ) w_2 \left(\frac{\vc x}{\phi(t)}\right)\d \vc x\, \d t  \\
 & \qquad = c_w   \int_S q_{n-m}^{(2m+d)}(t) q_{n'-m'}^{(2 m'+d)}(t) |\phi(t)|^{2m +d} w_1(t) \d t
            \int_{\BB^d} P_\kb^m(\vc u) P_{\kb'}^{m'}(\vc u) w_2(\vc u) \d \vc u \\
   &\qquad = h_m^n  \delta_{n,n'} \delta_{m,m'} \delta_{\kb,\kb'}.  
\end{align*}
Thus, $\{Q_{m,\kb}^n\}$ is an orthogonal set and they form a basis for $\CV_n(\partial\VV^{d+1})$. 
\end{proof}

 
Below we give several examples for specific quadratic domains of revolution. We shall always assume that
$\{P_\kb^m: |\kb| = m, \, \kb\in \NN_0^d\}$ denotes an orthonormal basis of $\CV_n^d(\varpi_\mu)$ for the
classical weight function \eqref{eq:w-ball} on the unit ball. 

\subsection{Cylinder}
We choose $\phi (t) = 1$. When $w_1(t) = t^{2\mu-1} w_{\a,\b}(t)$ on $[-1,1]$, where $w_{\a,\b}$ is the
Jacobi weight function and let $w_2 = \varpi_\mu$ be the classical weight function \eqref{eq:w-ball} on 
the unit ball. Then the weight function $W$ is given by 
$$
  W(\vc x,t) =  (1-t)^\a (1+t)^\b (1- \|\vc x\|^2)^{\mu-\f12},
$$
the orthogonal polynomials are given in terms of the Jacobi polynomials by 
$$
   Q_{m,\kb}^n(\vc x,t) = P_{n-m}^{(\a,\b)}(t) P_\kb^m(\vc x), \qquad  
$$
which are orthogonal on the cylinder in $\RR^{d+1}$,   
$$
   \VV^{d+1} =  \{(\vc x,t): \|\vc x\| \le 1, \,\, \vc x \in \RR^d, \,\, -1 \le t \le 1\} = \BB^d \times [-1,1]. 
$$

\subsection{Cone}
We choose $\phi (t) = t$. Let $w_1(t) = t^{2\mu-1} u_{\a,\b}(t)$ be defined on $(0,1)$, where $u_{\a,\b}$
is the Jacobi weight on $(0,1)$ in \eqref{eq:Jacobi01}, and let $w_2 = \varpi_\mu$. Then the weight function 
$W$ is given by 
$$
  W(\vc x,t) = t^{2\mu-1}u_{\a,\b}(t) \left (1-\frac{\|\vc x\|^2}{t^2} \right)^{\mu-\f12}
           = t^\a (1-t)^\b (t^2 - \|\vc x\|^2)^{\mu-\f12},  
$$
and the polynomials $Q_{m,\kb}^n$ in \eqref{eq:OPbasis} are given in terms of the Jacobi polynomials by
$$
  Q_{m,\kb}^{n,\mu}(\vc x,t) = P_{n-m}^{(2m+2\mu+\a+d-1,\b)}(1-2t) t^m
       P_{\kb}^m \left(\frac{\vc x}{t}\right), \quad |\kb| = m, \, 0\le m \le n,
$$
which are orthogonal on the solid cone in $\RR^{d+1}$, 
$$
   \VV^{d+1} = \{(\vc x,t): \|\vc x\| \le t, \,\, \vc x \in \RR^d, \,\, 0 \le t \le 1\}.
$$
When $\a = 0$, these polynomials are eigenfunctions of a second order differential operator \cite{X}.  

\subsection{Double cone}
We choose $\phi (t) = t$. Let $w_1(t)=|t|^{2\mu-1} (1-t^2)^{\b}$ be defined on $(-1,1)$, and let $w_2 = \varpi_\mu$. 
Then the weight function $W$ is given by 
$$
  W(\vc x,t) = t^{2\mu-1}(1-t^2)^\b \left (1-\frac{\|\vc x\|^2}{t^2} \right)^{\mu-\f12}
           =  (1-t^2)^\b (t^2 - \|\vc x\|^2)^{\mu-\f12},  
$$
 and the polynomials $Q_{m,\kb}^n$ in \eqref{eq:OPbasis} are given in terms of the generalized Gegenbauer
 polynomials associated to \eqref{eq:GGweight} by 
$$
  Q_{m,\kb}^n(\vc x,t) = C_{n-m}^{(\b+\f12, m+\mu+\f{d-1}{2})}(t) t^m P_{\kb}^m \left(\frac{\vc x}{t}\right), \quad |\kb| = m, \, 0\le m \le n,
$$
which are orthogonal on the double cone of $\RR^{d+1}$, 
$$
   \VV^{d+1} = \{(\vc x,t): \|\vc x\| \le |t|, \,\, \vc x \in \RR^d, \,\, -1 \le t \le 1\}.
$$

\subsection{Ball}
We choose $\phi(t) = \sqrt{1-t^2}$ for $t \in (0, 1)$. if $w_1(t) = (1-t^2)^{\mu-\f12}$ and $w_2=\varpi_\mu$ is 
the classical weight function on the unit ball, then 
$$
 W(\vc x,t) =  (1-t^2)^{\mu-\f12} \left (1-\frac{\|\vc x\|^2}{1-t^2} \right)^{\mu-\f12} =  (1- \|\vc x\|^2 - t^2)^{\mu-\f12}
$$
is the classical weight function on the unit ball in $\RR^{d+1}$. 

\subsection{Paraboloid of revolution}
We choose $\phi(t) = \sqrt{t}$ for $t \in (0, 1)$. If $w_1(t) = t^{\mu-\f12} u_{\a,\b}$ for $t\in (0, 1)$ and $w_2 = W_\mu$ is the classical 
weight function on the unit ball, then 
$$
   W(\vc x,t) =  t^{\mu-\f12} u_{\a,\b}(t) \left (1-\frac{\|\vc x\|^2}{t} \right)^{\mu-\f12} =  t^\a (1-t)^\b (t - \|\vc x\|^2)^{\mu-\f12},
$$
and the polynomials $Q_{m,\kb}^n$ in \eqref{eq:OPbasis} are given in terms of the generalized Gegenbauer
 polynomials by 
$$
  Q_{m,\kb}^n(\vc x,t) = P_{n-m}^{(m+\mu+\a+\f{d-1}{2},\b)}(1-2 t) t^{\f{m}{2}} P_{\kb}^m \left(\frac{\vc x}{\sqrt{t} }\right), \quad |\kb| = m, \, 0\le m \le n,
$$
which are orthogonal on the paraboloid of $\RR^{d+1}$, 
$$
   \VV^{d+1} = \left \{(\vc x,t): \|\vc x\|^2 \le t, \,\, 0 \le t \le 1, \,\, \vc x \in \RR^d \right \}.
$$
 
\subsection{Hyperboloid of revolution}
We choose $\phi(t) = \sqrt{t^2+ \rho^2}$ for $t \in (0, 1)$. If $w_1(t) =(t^2+\rho^2)^{\mu-\f12} w_{\a,\b}(t)$ for $t\in (-1, 1)$ and $w_2 = \varpi_\mu$ is the classical 
weight function on the unit ball, then 
$$
 W(\vc x,t) = (t^2+ \rho^2)^{\mu-\f12}  w_{\a,\b}(t)\left (1-\frac{\|\vc x\|^2}{t^2+\rho^2} \right)^{\mu-\f12} 
    = (1-t)^\a(1+t)^\b (t^2 + \rho^2 - \|\vc x\|^2)^{\mu-\f12}.
$$
The polynomial $q_{n-m}^{(2m+d)}$ that appears in $Q_{m,\kb}^n$ in \eqref{eq:OPbasis} need to be an orthogonal polynomial with respect to 
$$
    (1-t)^\a(1+t)^\b ({t^2+ \rho^2})^{m + \mu+ \frac{d-1}{2}} 
$$ 
on $[-1,1]$, which however cannot be written in a simple formula of the classical orthogonal polynomials. The polynomials $Q_{m,\kb}^n$ are
orthogonal on the hyperboloid of $\RR^{d+1}$
$$
   \VV^{d+1} =  \left \{(\vc x,t): \|\vc x\|^2 \le t^2+\rho^2,  \,\, a \le t \le b, \,\, \vc x \in \RR^d\right\}.
$$
 
\medskip

\section{Cubature rules on or in quadratic surfaces}
\setcounter{equation}{0}

We now construct cubature rules to discretize integrals on or in quadratic surfaces. Let $\Omega$ be a domain
in $\RR^{d}$. A cubature rule of degree $M$ for an integral $\int_\Omega f(\vc x) W(\vc x) d\vc x$ is a finite sum of 
function evaluations, 
$$
  \int_\Omega f(\vc x) W(\vc x) d\vc x = \sum_{k=1}^N \l_\k f(\vc x_k), \qquad \l_k \in \RR, \quad \vc x_k \in \RR^d, 
$$
where the equality holds if $f$ is any polynomial in $\Pi_M^d$. The numbers $\l_k$ are usually called 
weights of the cubature rule and $\vc x_k$ are called nodes. 


\subsection{Cubature rules on the surface}
Let $\phi: (a,b) \to \RR_+$, $b > a \ge 0$, be a weight function that satisfies either one of the assumptions in
Definition \ref{def:VV}. For the quadratic surface $\partial \VV^{d+1}$ defined by $\phi$, cubature rules on the 
surface can be constructed via cubature rules on the unit sphere and quadrature rules on the interval. 

First we need a cubature rule of degree $2n-1$ on the unit sphere. Such a rule takes the form
\begin{equation} \label{eq:cubaSph}
  \int_\sph f(\vc x) \d \s(\vc x) = \sum_{k=1}^{N_n} \mu_{k} f(\xi_k), \qquad \xi_k \in \sph, \quad \mu_k > 0,  
\end{equation}
and the equality holds for all $f \in \Pi_{2n-1}(\sph)$. In explicit construction, the sum in \eqref{eq:cubaSph} is
usually needed to be written as a multiple of sums. An explicit construction of such a cubature rule can be 
derived using spherical coordinates, which allows us to write the integral over $\sph$ as $(d-1)$ folds of integrals
of one variable, each can be approximated by a quadrature rule. The ensuing cubature rule is the usual product 
type cubature rule of degree $2n-1$ on $\sph$ (cf. \cite[Theorem 6.2.2]{DaiX}).  In this case, $N_n = n^d$. 

Let $p_n^\phi(t)$ be the orthogonal polynomials of degree $n$ with respect to the 
weight function $|\phi(t)|^{d-1}w(t)$ on $(a,b)$. It is known that $p_n^\phi$ has $n$ distinct zeros in $(a,b)$, 
which we denote by $t_{k}$, $1\le k \le n$. The Gauss quadrature for the integral $\int_a^b g(t) |\phi(t)|^{d-1}w(t) \d t$
is then given by  
\begin{equation} \label{eq:GaussQuad}
     \int_a^b g(t) |\phi(t)|^{d-1} w(t) \d t =  \sum_{k=1}^n \nu_{k} g(t_{k}), \qquad g \in \Pi_{2n-1},
\end{equation}
where $\Pi_m = \Pi_m^1$. It is known that $\nu_{k} > 0$ as it satisfies
$$
   \nu_{k} = \frac{1}{\sum_{k=1}^{n-1} [p_k(\phi;t_{k})]^2}.
$$ 

\begin{thm}\label{thm:sfCuba}
Let $\xi_k, \mu_{k}$ be as in \eqref{eq:cubaSph} and $t_j, \nu_{j}$ be as in \eqref{eq:GaussQuad}. Then
\begin{align} \label{eq:sfConeCuba}
  \int_{\partial \VV^{d+1}} f(\vc x,t) w(t) \d \s(\vc x,t) = \sum_{j=1}^n \nu_{j} \sum_{k=1}^{N_n} \mu_{k} 
        f(t_{j} \xi_k, t_{j})    
\end{align}
is a cubature rule of degree $2n-1$ for $w(t) \d \s(\vc x,t)$ on $\partial \VV^{d+1}$. 
\end{thm}

\begin{proof}
We verify that the cubature rules is exact for $f = S_{m,\ell}^r = q_{r-m}^{(m)} (t)Y_\ell^m (\vc x)$ in 
\eqref{eq:sfOPbasis} for $0 \le r \le 2n-1$. For such a $f$, the righthand side of \eqref{eq:sfConeCuba}
becomes
\begin{align} \label{eq:sfConeCuba1}
\sum_{j=1}^n \nu_{j} \sum_{k=1}^{N_n} \mu_{k}  S_{m,\ell}^r (t_{j} \xi_k, t_{j})
    =  \sum_{j=1}^n \nu_{j} \phi^m(t_{j})q_{r-m}^{(m)} (t_{j}) \sum_{k=1}^{N_n} \mu_{k} Y_\ell^m (\xi_k). 
\end{align}
If $m > 0$, then the last sum is zero by \eqref{eq:cubaSph} for $m \le r \le 2n-1$, which agrees with the 
integral of $S_{m,\ell}^r$ over $\partial \VV^{d+1}$ by the orthogonality of $S_{m,\ell}^r$. If $m =0$, the
sum is equal to 1, so that the righthand side of  \eqref{eq:sfConeCuba} is equal to 
$$ 
\sum_{j=1}^n \nu_{j}  q_{r}^{(0)} (t_{j}) = \int_a^b q_r^{(0)}(t) |\phi(t)|^{d-1} w(t) dt
   =   \int_{\partial \VV^{d+1}} S_{0,\ell}^r(\vc x,t) w(t) \d \s(\vc x,t),
$$
where the first identity follows from \eqref{eq:GaussQuad} since $q_r^{(0)}(t)$ is a polynomial of degree
$r \le 2n-1$, and the second identity follows from \eqref{eq:int-surface}. Putting together, we have verified 
\eqref{eq:sfConeCuba} for $f = S_{m,\ell}^r$ for $0\le m \le r \le 2n-1$ and all $\ell$. 
\end{proof}

For $d =2$, the surface in $\RR^3$, we have the following cubature rule
$$
  \int_{\partial \VV^3} f(\vc x,t) w(t) \d\sigma(\vc x,y)
    = \frac{\pi}{n} \sum_{j=1}^n \nu_{j} \sum_{k=0}^{2n-1}  f(t_{j} \cos \tfrac{k\pi}{n}, t_{j} \sin \tfrac{k\pi}{n}, t_{j}),  
$$
derived using the cubature rule on the circle with equally spaced points, which is exact for all trigonometric
polynomials of degree $2n-1$.

The cubature rule \eqref{eq:cubaSph} is not optimal and its nodes has a high concentration close to the 
origin. For example, consider $w(t) = t^\a(1-t)^\b$ on the surface of the cone in $\RR^3$, so that $\phi(t) = t$, 
then $t_{j}$ are zeros of the Jacobi polynomial $P_n^{(\a+2,\b)}(2t-1)$, which are all in $[0,1]$. In this
case, the nodes of the cubature rule \eqref{eq:sfConeCuba1} are concentrated at the apex of the cone. 

\subsection{Cubature rules on the solid domain}
Next we construct a cubature rule on the solid domain $\VV^{d+1}$ for the weight function
$$
W_\mu(\vc x,t) = w(t)  \left(\phi(t)^2 - \|\vc x\|^2\right)^{\mu - \frac12}, \qquad (\vc x,t) \in \VV^{d+1}, \quad \mu > -\tfrac12, 
$$
which is \eqref{eq:weightVV} with $w_2 = \varpi_\mu$, the classical weight function on the unit ball, and 
$w = w_1 |\phi|^{d+1-2\mu}$.

 We need cubature rules for $\varpi_\mu$ on the unit ball. Denote such a cubature rule of degree $2n-1$ by 
\begin{equation} \label{eq:cubaBall}
   \int_{\BB^d} f(\vc x) \varpi_\mu(\vc x) \d \vc x = \sum_{k=1}^{N_n} \mu_{k} f(\vc x_{k}), \qquad \vc x_k \in \BB^d, 
      \quad \mu_k > 0,  
\end{equation}
where the equality holds for all $f \in \Pi_{2n-1}^d$. An explicit construction of such a cubature rule can be 
derived either in cartesian product or in spherical polar coordinates, both allow us to write the integral over 
$\BB^d$ as $d$-folds of integrals of one variable so that the cubature rule can be written as product of 
cubature rules. In particular, using the identity 
$$
  \int_{\BB^d} f(\vc x) \varpi_\mu(\vc x) \d \vc x = \int_0^1 \int_\sph f(r \xi) \d \s(\xi) r^{d-1} (1-r^2)^{\mu-\f12} \d r,
$$
we can derive \eqref{eq:cubaBall} from that of \eqref{eq:cubaSph}. For example, let 
$s_{k} \in (0,1)$ be the zeros of the Jacobi polynomial $P_n^{(\mu-\f12, 0)}(2t^2-1)$ and
let $\delta_{k}$ be the weights of the Gaussian quadrature rule of degree $2n-1$ based on these zeros. 
Then 
\begin{equation} \label{eq:cubaBall2}
   \int_{\BB^2} f(\vc x) \varpi_\mu(\vc x) \d \vc x = \frac{\pi}{n} \sum_{k=1}^{n} \delta_{k}
     \sum_{j=0}^{2n-1} f \left(s_{k} \cos \tfrac{j\pi}{2n},s_{k} \sin \tfrac{j\pi}{2n}\right)  
\end{equation}
is a cubature rule of degree $2n-1$ for $\varpi_\mu$ on the disk $\BB^2$. 

We now need the Gauss quadrature for the integral $\int_a^b g(t) |\phi(t)|^{d}w(t) \d t$, which we denote
by 
\begin{equation} \label{eq:GaussQuad2}
     \int_a^b g(t) |\phi(t)|^d w(t) \d t =  \sum_{k=1}^n \nu_{k} g(t_{k}), \qquad g \in \Pi_{2n-1}. 
\end{equation}
This is the same as \eqref{eq:GaussQuad} but for the weight function $|\phi(t)|^d w(t)$. 

\begin{thm}\label{thm:coneCuba}
Let $\vc x_k, \mu_{k}$ be as in \eqref{eq:cubaBall} and $t_j, \nu_{j}$ be as in \eqref{eq:GaussQuad2}. Then
\begin{align} \label{eq:ConeCuba}
  \int_{\VV^{d+1}} f(\vc x,t) W_\mu (\vc x,t) \d \vc x \d t = \sum_{j=1}^n \nu_{j} \sum_{k=1}^{N_n} \mu_{k} 
        f(t_{j} \vc x_{k}, t_{j})    
\end{align}
is a cubature rule of degree $2n-1$ for $W_\mu$ on $\VV^{d+1}$. 
\end{thm}

\begin{proof}
We verify the cubature rule for the basis $Q_{m,\kb}^r$ defined in \eqref{eq:OPbasis} for $|\kb| = m \le r \le 2n-1$.
Using the identity \eqref{eq:Int-VV}, the proof follows along the line of the proof of Theorem \ref{thm:sfCuba}. We
omit the details. 
\end{proof}

For $d=2$, the cubature rule \eqref{eq:ConeCuba} is a triple sum, using, for example, the cubature rule in
\eqref{eq:cubaBall2}. As in the case of the surface of the cone, the nodes of the this cubature rule have a 
high concentration at the apex of the cone, inherited from the cubature rule on the unit ball.

\section{Function approximation on and inside a cone}

As an application, we consider the problem of function approximation using orthogonal polynomials on the surface and boundary of the cone. A cone is natural geometry that arises in applications; for example, the Taylor cone is used to describe jets of charged particles \cite{Tay}. 

\subsection{Function approximation on the surface of a cone in ${\mathbb R}^3$}
Here we consider approximations of the form:
$$
f(x,y,t) \approx \sum_{n=0}^N \sum_{m=0}^n \sum_{\ell=1}^2 f_{m,\ell}^n S_{m,\ell}^n(x,y,t) 
$$
where $S_{m,\ell}^n$ are OPs on the surface of a cone as defined in \eqref{eq:SurfaceConeDef} with $d = 2$ and parameters $\alpha = \beta = 0$ (for simplicity).  
The approach we take is to observe that the cosine and sine terms treated separately have precisely the same Jacobi polynomial parameters 
as OPs on the triangle, which allows us to use the transform introduced by Slevinsky in \cite{SlTr}, based on his 
spherical harmonics transform in \cite{SlSH}.  First we introduce the  lowering operator $L^{(a,b)}$:

\begin{prop}
Suppose $p$ is a degree $n$ polynomial satisfying $p(1) = 0$. Then  there is an $(n+1) \times n$ matrix $L^{(a,b)}$ satisfying
\begin{align*}
p(t) &=  (1-t) \left[P_0^{(a+2,b)}(t),\ldots,P_{n-1}^{(a+2,b)}(t)\right] \begin{pmatrix} p_0 \\ \vdots \\ p_{n-1} \end{pmatrix} \\
       &=  \left[P_0^{(a,b)}(t),\ldots,P_{n}^{(a,b)}(t)\right] L^{(a,b)} \begin{pmatrix} p_0 \\ \vdots \\ p_{n-1} \end{pmatrix}
\end{align*}
The range of $L^{(a,b)}$ is ${\rm ker}\big(\big[P_0^{(a,b)}(1), \ldots, P_n^{(a,b)}(1)\big]\big)$. 
The coefficients of $L^{(a,b)}$ are given in \cite[Theorem 5.1]{SlTr}.
\end{prop}

The transform consists of a two-stage process:

\def\cossin#1{\begin{cases}
\cos m #1 & \ell = 1 \\
\sin m #1 & \ell = 2
\end{cases}}
\begin{enumerate}
\item Expand in a tensor product basis: 
$$
f(x,y,t) \approx\sum_{k=0}^N \sum_{m=0}^N \sum_{\ell=1}^2c_{m,k,\ell}^N P_k^{(1,0)}(1-2t)  \cossin\t
$$
where   we assume $c_{0,k,2} = 0$ and $x = t \cos \theta$, $y = t \sin \theta$. 
This can be accomplished efficiently by quadrature on a tensor product grid of $N+1$ Gauss--Jacobi points 
and $2N+1$ evenly spaced points on the circle, $2N+1$ evenly spaced points on the circle:
\begin{align*}
c_{m,k,\ell}^N &:= {1 \over 2N+1} \sum_{\nu =1}^{N+1} \sum_{\eta=1}^{2N+1} w_\nu 
  f(t_\nu, \cos \theta_\eta, \sin \theta_\eta)\times \\
&\qquad   P_k^{(1,0)}(t_\nu) \cossin{\theta_\eta}
\end{align*}
where $(x_j,w_j)$ are the Gauss--Jacobi quadrature points and weights and $\theta_\eta = {2 \pi \eta \over 2N+1}$.  
Alternatively,  we can use the FFT and recent fast Jacobi transforms \cite{SlJac,TWO} which have quasi-optimal 
complexity. 
\item Convert from the tensor product basis to the orthogonal polynomial basis, by inverting:
$$
L \begin{pmatrix}
f_{m,\ell}^0 \\
\vdots \\
f_{m,\ell}^{N-m} 
\end{pmatrix}	=  \begin{pmatrix}
c_{0,m,\ell}^N \\
\vdots \\
c_{N,m,\ell}^N
\end{pmatrix}
$$
where $L := L^{(1,0)} L^{(3,0)}  \cdots L^{(2m-3,0)} L^{(2m-1,0)} $. 
\end{enumerate}

This procedure exactly recovers polynomials. Note this is distinct from polynomial interpolation as we oversample. 

\begin{lem}\label{lem:exactpoly}
Suppose $f(x,y,t)$ is a polynomial of degree $N$. Then the above procedure recovers the coefficients 
$\{f_{m,\ell}^n\}$ exactly. 	
\end{lem}

\begin{proof}
Linearity implies that it suffices to consider a single term $S_{m,\ell}^n(x,y,t)$. In the first stage, orthogonality 
guarantees that only one $m$  mode remains, and the exactness of quadrature ensures that 
$$
S_{m,\ell}^n(x,y,t) = \sum_{k=0}^N c_{m,k,\ell}^N P_k^{(1,0)}(t) Y_\ell^m(x,y).
$$
In other words
$$
 \sum_{k=0}^N c_{m,k,\ell}^N P_k^{(1,0)}(1-2t)  = t^m P_{n-m}^{(2m + 1,0)}(1-2t).
$$
This lies in the range of $L$ and therefore $L$ is  invertible.
\end{proof}

To leverage the existing code  for triangle transforms, we approach the second part of the transform 
by forming two functions on the triangle:
$$
f_\ell(x,y) :=  \sum_{m=0}^N \sum_{k=0}^N c_{m,k,\ell}^N P_k^{(1,0)}(x) T_m(y/(1-x)).
$$
Applying Slevinsky's transform computes
$$
f_\ell(x,y) = \sum_{n=0}^N \sum_{k=0}^n f_{k,\ell}^n P_{n-k}^{(2m+k+1,0)}(x) (1-x)^k T_k(y/(1-x))  .
$$
Slevinsky has shown that this algorithm is fast and stable, and achieves quasi-optimal complexity \cite{SlTr,SlSH}. 

\begin{figure}[tb] \label{fig:approxfun}
  \begin{center}
    \includegraphics[width=0.45\textwidth]{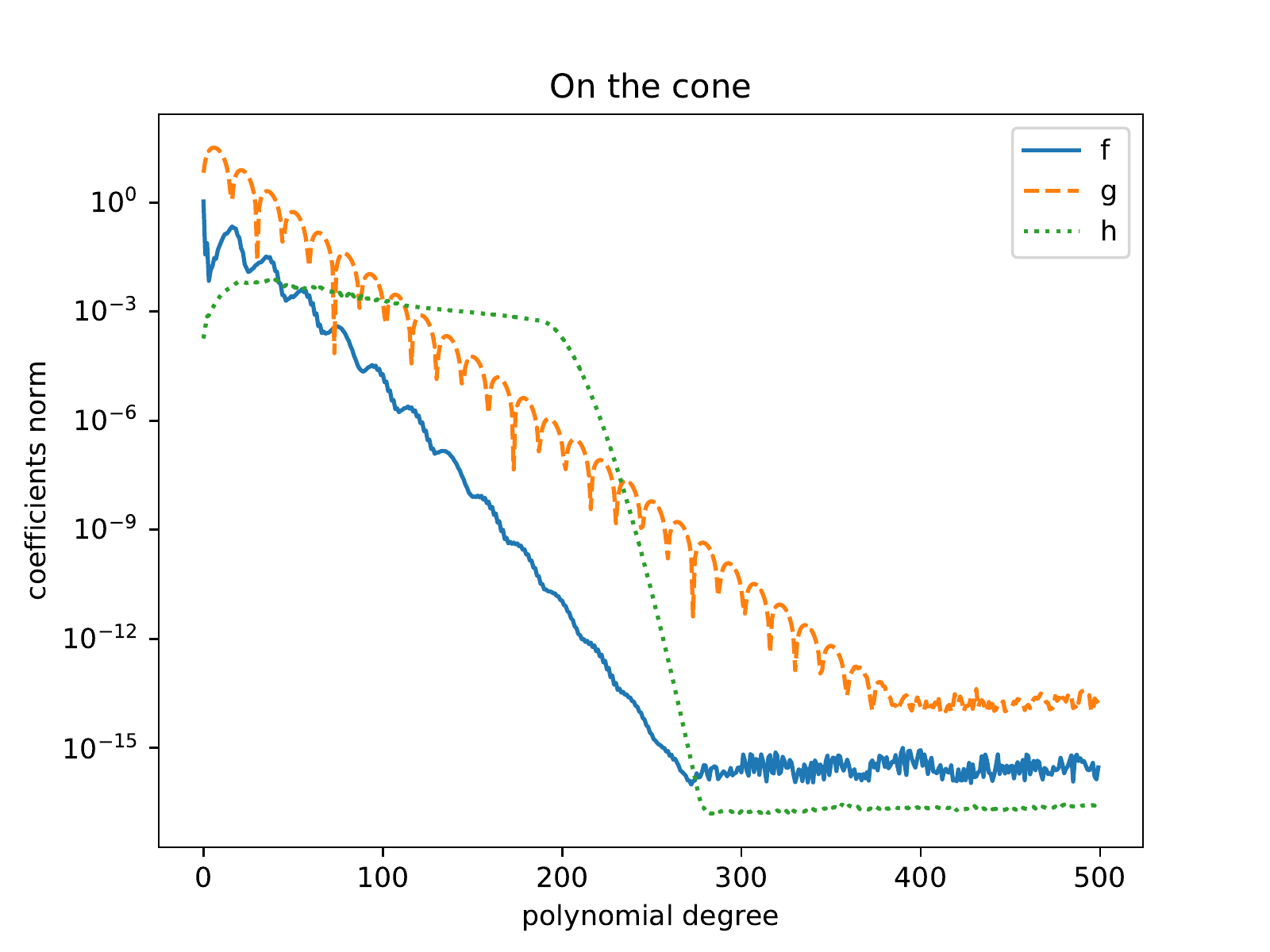}
    \includegraphics[width=0.45\textwidth]{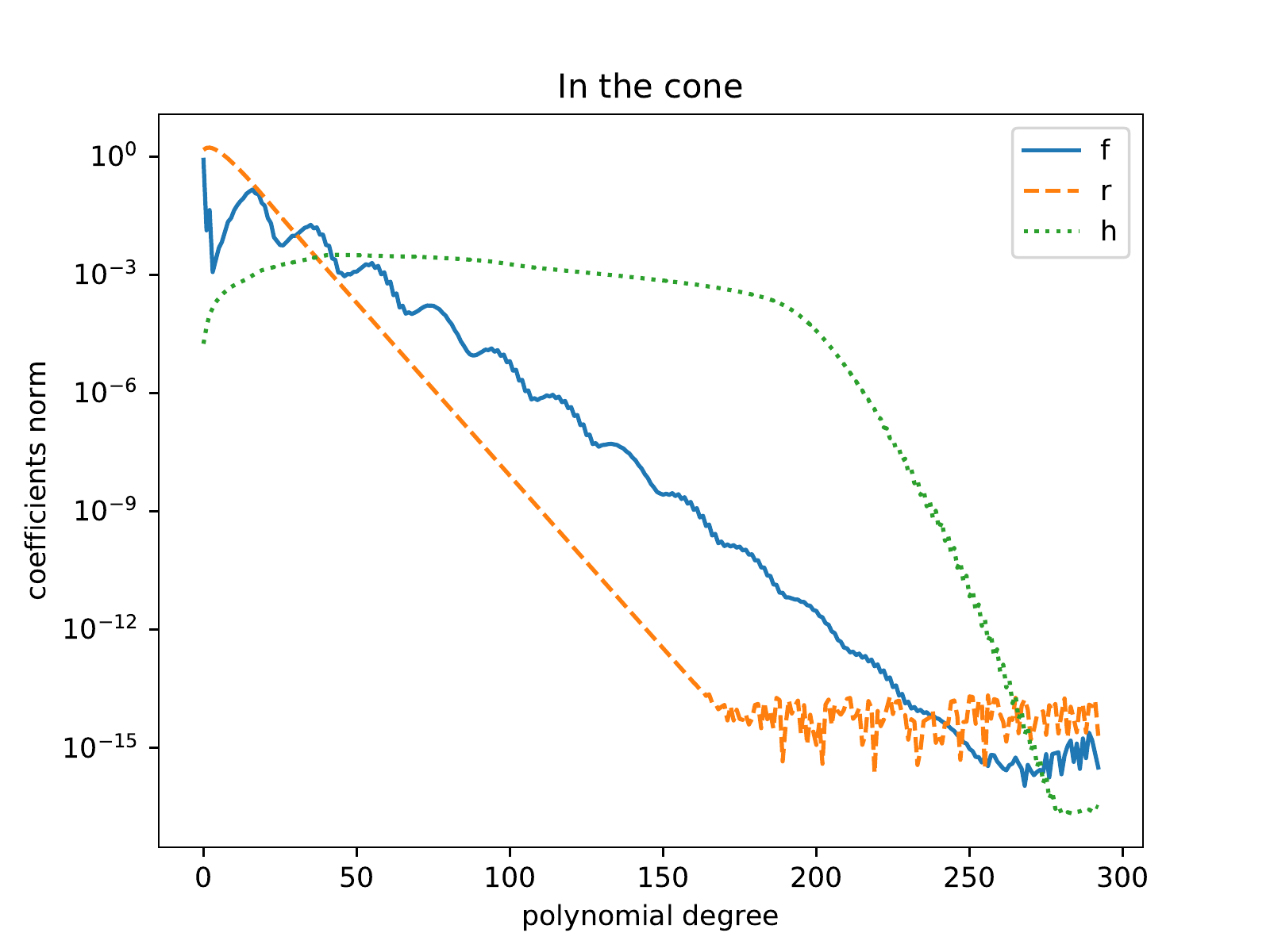}    
  \end{center}
  \caption{Decay of numerically calculated coefficients on (left) and in (right) the cone,  for $f(x,y,t) = \E^{x \cos(20y-t)}$, 
    $g(x,y,t) = {1 \over x^2 + y^2 + (t-0.02)^2}$, $h(x,y,t) = {\cos 100x(y-1) \over 1 + 50t}$, and     $r(x,y,t) = {1 \over x^2 + y^2 + (t+0.02)^2}$. We plot the 2-norm of the degree $n$ coefficients.   }
\end{figure}
 
As an application, consider approximation of a function on the disk of the form:
$$
F(x,y) = f\big(x,y,\sqrt{x^2 + y^2}\big)
$$
where $f(x,y,t)$ is smooth. We can employ the algorithm just constructed to efficiently expand $f$ on the cone, whereas attempting to approximate directly on the disk will have difficulties with the singularity at the origin. In Figure~\ref{fig:approxfun} we plot the decay of the norm of the degree-$n$ coefficients, that is, 
$$
\sqrt{\sum_{\ell=1}^2 \sum_{k=0}^n (f_{k,\ell}^n)^2 }
$$
 for three choices of $f$. Mimicking the univariate case, functions holomorphic in $x,y,t$ in a neighbourhood of the cone are observed to achieve exponential decay in their coefficients, with the location of the closest singularity dictating the rate of decay. Entire functions in $x,y,t$ are observed to achieve super-exponentially decaying coefficients. The rate of decay in the  coefficients is also indicative of the rate of pointwise convergence of the approximation. 

\subsection{Function approximation inside a cone in ${\mathbb R}^3$}

The procedure above extends to inside the cone with orthogonal polynomials in the ball replacing Fourier series, which are efficiently calculated using Slevinsky's transform \cite{SlTr}. We focus on the special case of $d = 2$ and $\a = \b = \mu = 0$.  We have a three-stage process:

\begin{enumerate}
\item Expand in a tensor product basis:
$$
f(x,y,t) \approx\sum_{k=0}^N \sum_{j=0}^N\sum_{m=0}^N \sum_{\ell=1}^2 c_{m,k,j,\ell}^N   P_k^{(2,0)}(t) T_j(r) \cossin\t
$$
where we assume $c_{0,k,j,2} = 0$ and $x = r \cos \theta$, $y = r \sin \theta$. 
This can be accomplished efficiently by quadrature on a tensor product grid of $N+1$ Gauss--Jacobi points (in $t$)  and $2N+1$ evenly spaced points on the circle,
\begin{align*}
c_{m,k,j,\ell}^N := {1 \over 2N+1} \sum_{\nu =1}^{N+1} \sum_{\eta=1}^{N+1} \sum_{\eta=1}^{2N+1} &w_\nu \omega_\eta f(t_\nu, r_\eta \cos \theta_\eta,r_\eta   \sin \theta_\eta)  \times \\
& P_k^{(2,0)}(t_\nu) T_j(r_\eta) \cossin{\theta_\eta}
\end{align*}
where $(t_\nu,w_\nu)$ are the Gauss--Jacobi quadrature points and weights, $(r_\eta,\omega_\eta)$ are Gauss--Chebyshev quadrature nodes and weights,  and $\theta_\ell = {2 \pi \ell \over 2N+1}$.  Alternatively, as in the surface case,  we can combine the FFT, DCT, and recent fast Jacobi transforms \cite{SlJac,TWO} which have quasi-optimal complexity. 
\item Use Slevinsky's fast transform for OPs on the disk \cite{SlTr} to re-expand:
$$
f(x,y,t) \approx\sum_{k=0}^N  \sum_{m = 0}^N \sum_{\ell=0}^m s_{m,k,\ell}^N   P_k^{(2,0)}(t) P_\ell^m(x,y)
$$
\item Convert from the tensor product basis to the orthogonal polynomial basis, by inverting:
$$
 L \begin{pmatrix}
f_{m,\ell}^{N-m} \\
\vdots \\
f_{m,\ell}^{N-m} 
\end{pmatrix}	=  \begin{pmatrix}
s_{0,m,\ell}^N \\
\vdots \\
s_{N,m,\ell}^N
\end{pmatrix}
$$
where $L := L^{(2,0)} L^{(4,0)}  \cdots L^{(2m-2,0)} L^{(2m,0)}$.  
\end{enumerate}

By the same argument as in Lemma~\ref{lem:exactpoly}, this recovers polynomials exactly.  The right-hand side of Figure~\ref{fig:approxfun} shows the norms of the numerically computed coefficients for three functions, also demonstrating the rate of decay is dictated by analyticity properties. 

\noindent {\bf Remark}: This procedure also works for the surface of a cone in ${\mathbb R}^4$, with Slevinsky's fast spherical harmonics transform \cite{SlSH}. 

\section{Conclusion} 

We have constructed orthogonal polynomials on and in quadratic bodies of revolution, which can be viewed as a generalisation of spherical harmonics and orthogonal polynomials in the ball. We have demonstrated that these can be used effectively to construct quadrature rules and to approximate functions. 

\end{document}